\documentclass[12pt]{amsart}
\usepackage{amsmath,amssymb}

\usepackage{calrsfs}
\usepackage{url}
\newtheorem{thm}{Theorem}[section]
\newtheorem{cor}[thm]{Corollary}
\newtheorem{lem}[thm]{Lemma}
\newtheorem{prop}[thm]{Proposition}

\theoremstyle{definition}

\newtheorem{exmp}[thm]{Example}

\newtheorem{rem}[thm]{Remark}
\numberwithin{equation}{section}

\newcommand{\Z}{{\mathbb{Z}}}
\newcommand{\Q}{{\mathbb{Q}}}

\newcommand{\F}{{\mathbb{F}}}
\newcommand{\N}{{\mathbb{N}}}

\newcommand{\fS}{{\mathfrak{S}}}

\newcommand{\cD}{{\mathcal{D}}}
\newcommand{\cE}{{\mathcal{E}}}

\newcommand{\cI}{{\mathcal{I}}}
\newcommand{\cH}{{\mathcal{H}}}
\newcommand{\cM}{{\mathcal{M}}}
\newcommand{\cN}{{\mathcal{N}}}

\newcommand{\bq}{{\boldsymbol{q}}}
\DeclareMathOperator{\Irr}{Irr}
\renewcommand{\leq}{\leqslant}
\renewcommand{\geq}{\geqslant}

\begin{document}
\title[Green functions and Glauberman degree-divisibility]{Green
functions and Glauberman degree-divisibility}
\author[Meinolf Geck]{Meinolf Geck}
\address{IAZ - Lehrstuhl f\"ur Algebra, Universit\"at Stuttgart, 
Pfaffenwaldring 57, D--70569 Stuttgart, Germany}
\email{meinolf.geck@mathematik.uni-stuttgart.de}

\date{April 4, 2019}

\begin{abstract} The Glauberman correspondence is a fundamental bijection
in the character theory of finite groups. In 1994, Hartley and Turull 
established a degree-divisibility property for characters related by
that correspondence, subject to a congruence condition which should hold 
for the Green functions of finite groups of Lie type, as defined by Deligne 
and Lusztig. Here, we present a general argument for completing the proof 
of that congruence condition. Consequently, the degree-divisibility
property holds in complete generality.
\end{abstract}

\maketitle

\section{Introduction} \label{sec0}

This paper is mainly about representations of finite groups of Lie type, 
but the motivation comes from the general character theory of
finite groups. Let $\Gamma$, $S$ be finite groups of coprime order such 
that $S$ is solvable and acts by automorphisms on $\Gamma$. Then the 
Glauberman correspondence \cite{Glau} is a certain canonical bijection
\[ \Irr_S(\Gamma)\leftrightarrow \Irr(C_\Gamma(S)),\qquad \theta 
\leftrightarrow \theta^*,\] 
where $\Irr_S(\Gamma)$ is the set of $S$-invariant irreducible characters 
of $\Gamma$ and $C_\Gamma(S)$ the subgroup of $\Gamma$ fixed by all elements 
of $S$. It is of fundamental importance in various current trends of 
research; see, e.g., Navarro \cite[\S 2.5]{Nav}. Using the classification 
of finite simple groups, and subject to a certain congruence condition on 
Green functions of finite groups of Lie type, Hartley--Turull
\cite[Theorem~A]{HaTu} showed the following result, which gives a positive
answer to a problem described as perhaps one of the deepest in character 
theory by Navarro \cite[\S 1]{Nav}. 

\medskip
\noindent {\bf Glauberman degree-divisibility}. \textit{Assume
that $\theta \in \Irr_S(\Gamma)$ and $\theta^*\in \Irr(C_\Gamma(S))$ 
correspond to each other as above. Then $\theta^*(1)$ divides $\theta(1)$}. 
\medskip

That congruence condition on Green functions has been explicitly 
verified in \cite[Prop.~7.5]{HaTu} for groups of type $A_n$ and in 
\cite[Prop.~7.7]{HaTu} for groups of type ${^3\!D}_4$, ${^2\!B}_2$, 
${^2\!G}_2$, ${^2\!F}_4$; by the comments on \cite[p.~204]{HaTu} it is also
known in a large number of further cases. In this paper, we present a general 
argument which completes the proof of that condition. Hence, the Glauberman 
degree-divisibility property will hold unconditionally and in complete 
generality.

We shall use the full power of the geometric representation theory of finite 
groups of Lie type, as developed by Lusztig \cite{L1}, \cite{L2a}--\cite{L2e},
\cite{L5}, \cite{L10}; an essential role will also be played by the results 
of Shoji \cite{S2}, \cite{S3} concerning the relation between irreducible 
representations and character sheaves. 

Let us now explain that congruence condition on Green functions. We consider
a connected reductive algebraic group $G$ (over an algebraic closure of 
$\F_p$ where $p$ is a prime) and an endomorphism $F\colon G\rightarrow G$ such
that some power of $F$ is a Frobenius map. The Green function $Q_T$ 
corresponding to an $F$-stable maximal torus $T\subseteq G$ is introduced by 
Deligne--Lusztig \cite{DeLu} (see also Carter \cite{C2}). It is a function 
defined on the set of unipotent elements of $G^F$, with values in $\Z$. The
construction involves the theory of $\ell$-adic cohomology applied to certain 
algebraic varieties on which the finite group~$G^F$ acts. In order to 
indicate the dependence on~$F$, we shall write $Q_{T,F}$ instead of 
just~$Q_T$. Now we can state:

\medskip
\noindent {\bf Congruence Condition} (Hartley--Turull \cite[Condition~6.9, 
6.10]{HaTu}). \textit{Let $T\subseteq G$ be an $F$-stable maximal torus 
and $u\in G^F$ be unipotent. Let $r\in \N$ be a prime such that 
$r$ does not divide the order of $|G^{F^r}|$. Then}
\[ Q_{T,F}(u)\equiv Q_{T,F^r}(u) \bmod r.\]

\medskip
A quick informal argument to establish this condition goes as follows. 
Since Suzuki and Ree groups have already been dealt with, we can assume 
that $F$ is a Frobenius map defining an $\F_q$-rational structure 
on $G$, where $q$ is a power of~$p$. It is expected that $Q_{T,F}(u)$ is 
given by a well-defined polynomial in~$q$ with integer coefficients, such 
that $Q_{T,F^r}(u)$ is given by evaluating that same polynomial at~$q^r$. 
Then it simply remains to use Fermat's Little Theorem. For example, 
this works perfectly well if $G$ is of type $A_n$, as already noted in 
\cite[Prop.~7.5]{HaTu}, and in many further cases; see Shoji 
\cite[\S 6]{S1}. However, the required information is not yet available for
all groups over fields of small characteristic. And 
even when it is known, then some additional care is
needed since there are cases where the Green functions are only ``PORC'' (in
the sense of Higman), that is, \underline{p}olynomial \underline{o}n 
\underline{r}esidue \underline{c}lasses of~$q$; see Beynon--Spaltenstein 
\cite{BeSp}. Thus, it seems desirable to find a general argument, uniformly 
for all characteristics~$p$ and appropriately dealing with the 
``PORC'' phenomenon~---~and this is what we will do in this paper.

It was first shown by Lusztig \cite{L5} (with some mild restrictions 
on~$q$) and then by Shoji \cite{S2}, \cite{S3} (in complete generality) 
that the original Green functions of \cite{DeLu} can be identified with 
another type of Green functions defined in terms of Lusztig's character 
sheaves \cite{L2b}. This provides new, extremely powerful tools. 

In Section~\ref{sec1}, we review the general plan for determining the Green 
functions, taking into account the above developments. The main result of 
Section~\ref{sec2} is Theorem~\ref{thmsplit}, which is inspired by 
\cite{pbad} and implies a crucial ``PORC'' property. In Section \ref{sec3} 
we recall the basic ingredients of the Lusztig--Shoji algorithm which reduces
the computation of the Green functions to the determination of certain signs. 
Our Theorem~\ref{thmsplit} does not determine these signs, but it ensures 
that the signs behave well with respect to replacing $F$ by~$F^r$. In 
Section~\ref{sec4} we put all these pieces together to complete the proof 
of the Congruence Condition.

We only mention here that Theorem~\ref{thmsplit} is also useful in another 
direction, for computational purposes: in \cite{small} it is the main 
theoretical tool to complete the computation of Green functions for several 
cases where these have not been previously known (e.g., type $E_7$ in 
characteristic $2,3$). 

We assume some general familiarity with the theory of finite groups of Lie 
type and the character theory of these groups; see, e.g., \cite{C2}, 
\cite{cbms}, \cite{first}. 

\medskip
\textit{Acknowledgements}. The author is indebted to Gunter Malle and Jay 
Taylor for pointing out the open question about Green functions in the article
of Hartley--Turull \cite{HaTu} at the Oberwolfach workshop ``Representations
of finite groups'' in March 2019. I am also grateful to Gabriel Navarro for 
helpful comments, and to Jay Taylor for a careful reading of the paper. This 
work is a contribution to the SFB-TRR 195 ``Symbolic Tools in Mathematics 
and their Application'' of the German Research Foundation (DFG). 

\section{Green functions and character sheaves} \label{sec1}

Let $p$ be a prime and $k=\overline{\F}_p$ be an algebraic closure of the
field with $p$ elements. Let $G$ be a connected reductive algebraic group 
over $k$ and assume that $G$ is defined over the finite subfield $\F_q
\subseteq k$, where $q=p^m$ for some $m\geq 1$. Let $F\colon G\rightarrow 
G$ be the corresponding Frobenius map. Let $B_0 \subseteq G$ be an 
$F$-stable Borel subgroup and $T_0\subseteq B_0$ be an $F$-stable maximal
torus. Let $W=N_G(T_0)/T_0$ be the corresponding Weyl group. For each 
$w\in W$, let $R_w$ be the virtual representation of the finite group 
$G^F$ defined by Deligne--Lusztig \cite[\S 1]{DeLu}. (In the setting 
of \cite[\S 7.2]{C2}, we have $\mbox{Tr}(g,R_w)=R_{T_w,1}(g)$ for $g\in G^F$,
where $T_w\subseteq G$ is an $F$-stable maximal torus obtained from $T_0$ 
by twisting with~$w$, and $1$ stands for the trivial character of $T^F$.) 
This construction is carried out over $\overline{\Q}_\ell$, an algebraic
closure of the $\ell$-adic numbers where $\ell$ is a prime not equal to~$p$. 
The corresponding Green function is defined by 
\[ Q_w\colon G_{\text{uni}}^F \rightarrow \overline{\Q}_\ell,
\qquad u \mapsto \mbox{Tr}(u,R_w),\]
where $G_{\text{uni}}$ denotes the set of unipotent elements of $G$.
It is known that $Q_w(u)\in \Z$ for all $u\in G_{\text{uni}}^F$;
see \cite[\S 7.6]{C2}. So the character formula \cite[7.2.8]{C2} shows 
that we also have $\mbox{Tr}(g,R_w)\in \Z$ for all $g\in G^F$.

The general plan for determining the values of $Q_w$
is explained in Lusztig \cite[Chap.~24]{L2e} and Shoji \cite[\S 5]{S1}, 
\cite[1.1--1.3]{S6a} (even for generalised Green functions, which we will 
not consider here). We will have to go through some of the steps of that plan.

\subsection{Almost characters} \label{sub21}

The Frobenius map $F$ induces an automorphism of $W$ which we denote by 
$\gamma\colon W\rightarrow W$. Let $\Irr(W)$ be the set of irreducible 
representations of $W$ over $\overline{\Q}_\ell$ (up to isomorphism). Let 
$\Irr(W)^\gamma$ be the set of all those $E\in \Irr(W)$ which are 
$\gamma$-invariant, that is, there exists a bijective linear map $\sigma_E
\colon E\rightarrow E$ such that $\sigma_E {\circ} w=\gamma(w){\circ} \sigma_E
\colon E\rightarrow E$ for all $w\in W$. Note that $\sigma_E$ is only 
unique up to scalar multiples but, if $\gamma$ has order $d\geq 1$, then
one can always find some $\sigma_E$ such that 
\[\sigma_E^d=\mbox{id}_E\qquad\mbox{and}\qquad \mbox{Tr}(\sigma_E{\circ} w,
E) \in \Z\quad\mbox{for all $w\in W$};\]
see \cite[3.2]{L1}. In what follows, we assume that a fixed choice of 
$\sigma_E$ satisfying the above conditions has been made for each $E\in
\Irr(W)^\gamma$. (For example, one could take the ``preferred'' choice for 
$\sigma_E$ specified by Lusztig \cite[17.2]{L2d}.) For $E\in \Irr(W)^\gamma$,
the corresponding almost character is the class function $R_{E}\colon 
G^F\rightarrow \overline{\Q}_\ell$ defined by  
\[ R_{E}(g):=\frac{1}{|W|} \sum_{w\in W} \mbox{Tr}(\sigma_E {\circ} w,
E)\mbox{Tr}(g,R_w) \qquad \mbox{for all $g\in G^F$}.\]
We have $R_{E}(g)\in \Q$ for all $g\in G^F$. By \cite[3.9]{L1}, the functions 
$R_E$ are orthonormal with respect to the standard inner product on class 
functions of $G^F$. Furthermore, by \cite[3.19]{cbms}, we have 
\[ Q_w(u)=\sum_{E\in \Irr(W)^\gamma} \mbox{Tr}(\sigma_E 
{\circ} w,E) R_{E}(u)\qquad \mbox{for $w\in W$, $u\in 
G_{\text{uni}}^F$}.\]
Hence, knowing the values of all Green functions $Q_w$ is equivalent to 
knowing the values of all $R_{E}$ on $G_{\text{uni}}^F$.

\subsection{Constructible $\overline{\Q}_\ell$-sheaves} \label{sub22}

Let $\cD G$ be the bounded derived
category of constructible $\overline{\Q}_\ell$-sheaves on $G$ (in the 
sense of Beilinson, Bernstein, Deligne \cite{bbd}), which are equivariant 
for the action of $G$ on itself by conjugation. The ``character sheaves'' 
on $G$, defined by Lusztig \cite{L2a}, all belong to $\cD G$. 
Consider any object $A\in \mathcal{D}G$ and suppose that its inverse 
image $F^*A$ under the Frobenius map is isomorphic to $A$ in $\mathcal{D}G$.
Let $\phi \colon F^*A \stackrel{\sim}{\rightarrow} A$ be an isomorphism. 
Then $\phi$ induces a linear map $\phi_{i,g}\colon \cH_{F(g)}^i(A)
\rightarrow \cH_g^i(A)$ for each $i$ and $g\in G$, where $\cH^i(A)$ denotes 
the $i$-th cohomology sheaf of $A$ and $\cH_g^i(A)$ the stalk at~$g\in G$. 
By \cite[8.4]{L2b}, this gives rise to a class function $\chi_{A,\phi} 
\colon G^F\rightarrow \overline{\Q}_\ell$, called a ``characteristic 
function'' of~$A$, defined by 
\[\chi_{A,\phi}(g)=\sum_i (-1)^i \mbox{Tr}(\phi_{i,g},\cH_g^i(A))\qquad
\mbox{for $g \in G^F$}.\] 
Note that, by a version of Schur's Lemma, $\phi$ is 
unique up to a non-zero scalar; hence, $\chi_{A,\phi}$ is unique up to a
non-zero scalar. If $F^*A\cong A$, then one can choose an isomorphism 
$\phi_A \colon F^*A \stackrel{\sim}{\rightarrow} A$ such that the values of
$\chi_{A,\phi_A}$ are cyclotomic integers and the standard inner 
product of $\chi_{A,\phi_A}$ with itself is equal to~$1$. The precise 
conditions which guarantee these properties are formulated in 
\cite[13.8]{L2c}, \cite[25.1]{L2e}; note that these conditions specify
$\phi_A$ up to multiplication by a root of unity. 

\subsection{The complexes $A_E$} \label{sub23}

Lusztig \cite[\S 8.1]{L2b} describes a geometric induction process by 
which one obtains objects in $\cD G$ from objects in $\cD L$ where $L$ is 
a Levi subgroup of some parabolic subgroup of $G$. Applying this to $L=T_0$
and the constant local system $\overline{\Q}_\ell$ on $T_0$, we obtain a 
well-defined complex $K\in \cD G$ together with a canonical isomorphism 
$\varphi\colon F^*K\stackrel{\sim}{\rightarrow} K$. The restriction of the
corresponding characteristic function $\chi_{K,\varphi}\colon G^F\rightarrow 
\overline{\Q}_\ell$ to $G_{\text{uni}}^F$ is an example of a ``generalised
Green function'', as defined in \cite[\S 8.3]{L2b}; see also \cite[1.7]{S2}. 
We have $\mbox{End}(K) \cong \overline{\Q}_\ell[W]$ (see \cite[24.2]{L2b}) 
and $K$ has a canonical decomposition
\[ K\cong \bigoplus_{E\in \Irr(W)} V_E\otimes A_E,\]
where $A_E\in \cD G$ is a character sheaf and $V_E=\mbox{Hom}(A_E,K)$ 
is an irreducible $W$-module isomorphic to $E\in \Irr(W)$; see 
\cite[1.2]{S6a}. Now let $E\in \Irr(W)^\gamma$. Then we also have $F^*A_E
\cong A_E$ and, using our choice of $\sigma_E\colon E\rightarrow E$ in 
\S \ref{sub21}, we can single out a particular isomorphism $\phi_{A_E}
\colon F^*A_E \stackrel{\sim}{\rightarrow} A_E$ as in \S \ref{sub22}. 
Since this will be important later, let us briefly indicate how this is 
done, following \cite[24.2]{L2e} or \cite[1.3]{S6a}. We start with any 
isomorphism $\phi_{A_E} \colon F^*A_E \stackrel{\sim}{\rightarrow} A_E$. 
Then there is a unique linear map $\psi_E\colon V_E \rightarrow V_E$ such 
that $\phi_{A_E} \otimes \psi_E$ corresponds to $\varphi\colon F^*K 
\stackrel{\sim}{\rightarrow} K$ under the above direct sum decomposition; 
see \cite[10.4]{L2b}. Furthermore, by \cite[24.2]{L2e}, $\psi_E$ 
corresponds under a $W$-module isomorphism $V_E\cong E$ to 
a non-zero scalar $\zeta\in \overline{\Q}_\ell$ times the map $\sigma_E\colon 
E\rightarrow E$. Hence, replacing $\phi_{A_E}$ by a scalar multiple, 
we can achieve that $\zeta=1$. Having fixed this choice of $\phi_{A_E}
\colon F^*A_E \stackrel{\sim}{\rightarrow} A_E$, let $\chi_{A_E} 
\colon G^F\rightarrow \overline{\Q}_\ell$ be the corresponding characteristic 
function. Then, by the main result of Lusztig \cite{L5} and by Shoji 
\cite[Theorem~5.5]{S3} (see also the argument in \cite[2.17, 2.18]{S2}), 
we have 
\[ R_{E}(g)=(-1)^{\dim T_0}\chi_{A_E}(g)\qquad\mbox{for all $g\in G^F$}.\]
(In \cite[2.18]{S2}, it is assumed that $q$ is a sufficiently
large power of~$p$, but this condition is later removed thanks to 
\cite[Theorem~5.5]{S3}.) The above identity is a special case of a more 
general conjecture about the relation between almost characters and 
characteristic functions of character sheaves; see \cite[p.~226]{L2b}, 
\cite{S2}, \cite{S3}.

\subsection{The Springer correspondence} \label{sub24}

Let $\cN_G$ be the set of all pairs $(C,\cE)$ where $C$ is a unipotent
class in $G$ and $\cE$ is a $G$-equivariant irreducible 
$\overline{\Q}_\ell$-local system on $C$ (up to isomorphism). The Springer
correspondence defines an injective map 
\[ \iota_G\colon \Irr(W)\hookrightarrow \cN_G\]
such that, if $E\in \Irr(W)$ and $\iota_G(E)=(C,\cE)$, then we have 
\[ \mbox{u-supp}(A_E)\subseteq \overline{C} \qquad \mbox{and}\qquad
\cH^i(A_E)|_C\cong \left\{\begin{array}{cl} \cE & \mbox{ if
$i=-\dim C-\dim T_0$},\\ 0 & \mbox{ otherwise}.\end{array}\right.\]
See Lusztig \cite{LuIC}, \cite[Chap~24]{L2e}, and the references there. 
Here, $\mbox{u-supp}(A)$ for any $A\in \cD G$ is defined as the Zariski 
closure of 
\[\{g\in G_{\text{uni}}\mid \cH_g^i(A) \neq \{0\} \mbox{ for some $i$}\}.\]
Given $E\in \Irr(W)$ and $\iota_G(E)= (C,\cE)$, we define 
\[ d_E:=(\dim G-\dim C-\dim T_0)/2.\]
Note that $\dim C_G(g)\geq \dim T_0$ for $g\in G$. Furthermore,
$d_E\in\Z_{\geq 0}$ since $\dim G-\dim T_0$ is always even and so is 
$\dim C$; see  \cite[\S 5.10]{C2} and the references there. 

\subsection{The $Y$-functions} \label{sub25}

Let $E\in \Irr(W)^\gamma$ and $\iota_G(E)=(C,\cE)$. Then $F(C)=C$ and
$F^*\cE\cong \cE$.  Since $\cE\cong \cH^i(A_E)|_C$ for $i=-\dim C-\dim T_0$,
the isomorphism $\phi_{A_E}\colon F^*A_E \stackrel{\sim}{\rightarrow} A_E$ 
induces a map $F^*\cE \stackrel{\sim}{\rightarrow} \cE$ which we can write 
as $q^{d_E}\psi$ where $\psi\colon F^*\cE\stackrel{\sim}{\rightarrow} \cE$ 
is an isomorphism.  With this normalisation, $\psi$ induces an automorphism
of finite order $\psi_g\colon \cE_g\rightarrow \cE_g$ at each stalk $\cE_g$ 
where $g\in C^F$; see \cite[24.2.4]{L2e}. For $g\in C^F$ and $i=
-\dim C -\dim T_0$, we have 
\[\chi_{A_E}(g)=(-1)^{i} \mbox{Tr}\bigl(\phi_{A_E,i,g}, \cH^{i}_g(A_E)
\bigr)= (-1)^{i}q^{d_E}Y_E(g)\]
where the class function $Y_E\colon G_{\text{uni}}^F\rightarrow 
\overline{\Q}_\ell$ is defined by 
\[ Y_E(g)=\left\{\begin{array}{cl} \mbox{Tr}(\psi_g,\cE_g) & \quad
\mbox{if $g \in C^F$}, \\ 0 & \quad \mbox{otherwise};\end{array}\right.\]
see \cite[24.2.3]{L2e}. In particular, the values of $Y_E$ are
algebraic integers. Since $\dim C$ is an even number, we obtain that 
\[ R_{E}(g)=(-1)^{\dim T_0}\chi_{A_E}(g)=q^{d_E}Y_E(g)
\qquad \mbox{for all $g\in C^F$}.\]
Since the values of $R_E$ are rational numbers (see \S \ref{sub21}), we 
conclude that
\[ Y_E(g)\in\Z \qquad \mbox{for all $E\in\Irr(W)^\gamma$ and 
$g\in G_{\text{uni}}^F$}.\]
The $Y$-functions $\{Y_E\mid E\in \Irr(W^\gamma)\}$ are linearly 
independent by \cite[24.2.7]{L2e}. 

\subsection{The coefficients $p_{E',E}$} \label{sub26} Having established 
the above framework, Lusztig \cite[Theorem~24.4]{L2e} shows that we have 
unique equations 
\[ R_{E}|_{G_{\text{uni}}^F}=\sum_{E'\in \Irr(W)^\gamma}  q^{d_E}\,
p_{E',E} Y_{E'}\qquad\mbox{for all $E\in \Irr(W)^\gamma$},\]
where the coefficients $p_{E',E}\in \overline{\Q}_\ell$ are determined by a 
purely combinatorial algorithm which we will consider in more detail
in Section~\ref{sec3}. Note that the hypotheses of 
\cite[Theorem~24.4]{L2e} (``cleanness'') are always satisfied by the main 
result of \cite{L10}. (Since we are only dealing with Green functions 
of $G^F$, and not with generalised Green functions, it would actually be 
sufficient to refer to \cite[\S 3]{aver} instead of \cite{L10}.)
By \cite[24.5.2]{L2e}, we have
\[ p_{E',E}\in \Z \qquad \mbox{for all $E,E'\in \Irr(W)^\gamma$}.\]
Furthermore, by \cite[24.2.10, 24.2.11]{L2e}, we have 
\[ p_{E,E}=1 \qquad\mbox{and}\qquad p_{E',E}=0 \quad \mbox{if $E'\neq E$ 
and $d_{E'}\geq d_E$}.\] 
Consequently, for a suitable ordering of $\Irr(W)^\gamma$, the matrix of 
coefficients $(p_{E',E})$ will be triangular with $1$ along the diagonal
(see also Section~\ref{sec3}). 

\medskip
\textit{Thus, the whole problem of computing the Green 
functions $Q_w$ of $G^F$ is reduced to the determination of the functions 
$\{Y_E\mid E \in \Irr(W)^\gamma\}$} (cf.\ Shoji \cite[1.3]{S6a}). 

\medskip
As in \cite[1.3]{S6a}, the above discussion also applies, with additional
technical complications, to the generalized Green functions defined in 
\cite[\S 8.3]{L2b}, but in this article we restrict ourselves to the 
``ordinary'' Green functions $Q_w$.

\section{Evaluating the $Y$-functions} \label{sec2}

Combining and summarizing the formulae in Section~\ref{sec1}, we can state 
the following result about the values of the Green functions of $G^F$. 

\begin{prop} \label{sumup} Let $w\in W$ and $u \in G^F$ be unipotent. Then 
\[ Q_w(u) =\sum_{E',E\in \Irr(W)^\gamma} \operatorname{Tr}(\sigma_E{\circ} w,
E) q^{d_E} p_{E',E} Y_{E'}(u),\]
where $\sigma_E,d_E,Y_{E'},p_{E',E}$ are defined in \S \ref{sub21}, 
\S \ref{sub24}, \S \ref{sub25}, \S \ref{sub26}, respectively.
\end{prop}

\begin{proof} By \S \ref{sub21}, we can express $Q_w(u)$ as a linear
combination of $R_E(u)$, for various $E\in \Irr(W)^\gamma$. By
\S \ref{sub26}, we can express each term $R_E(u)$ as a linear combination
of $Y_{E'}(u)$, for various $E'\in \Irr(W)^\gamma$. 
\end{proof}
 
In this section, we address the further evaluation of the terms $Y_{E'}(u)$. 
As we will use results from Shoji \cite{S2}, \cite{S3} we will assume from 
now on that the Frobenius map $F\colon G\rightarrow G$ is given by 
\[ F=\tilde{\gamma}\circ F_p^m=F_p^m \circ \tilde{\gamma}\qquad
(m\geq 1)\]
where $\tilde{\gamma}\colon G \rightarrow G$ is an automorphism of finite
order leaving $T_0,B_0$ invariant, and $F_p\colon G \rightarrow G$ is
a Frobenius map corresponding to a split $\F_p$-rational structure,
such that $F_p(t)=t^p$ for all $t\in T_0$. Note that $F_p$ acts trivially
on $W$ and that $\tilde{\gamma}$ induces an automorphism of $W$ which is 
just the automorphism $\gamma\colon W\rightarrow W$ induced by $F$ considered
earlier. Thus, if $G$ is semisimple, then $G^F$ is an untwisted or twisted 
Chevalley group, as in Steinberg \cite[\S 11.6]{St68}. 

\begin{rem} \label{rem20} It is known that all unipotent classes of $G$ 
are $F_p$-stable (since, in each case, representatives of the classes are 
known which lie in $G^{F_p}=G(\F_p)$; see, e.g., Liebeck--Seitz \cite{LiSe}).
Let $C$ be an $F$-stable unipotent class. We shall also make the following 
assumption.
\begin{itemize}
\item[($\clubsuit$)] There exists an element $u_0\in C^F$ such that $F$ 
acts trivially on the finite group of components $A(u_0):=
C_G(u_0)/C_G^\circ(u_0)$.
\end{itemize}
If ($\clubsuit$) holds, then there is a bijective correspondence between 
the conjugacy classes of $A(u_0)$ and the conjugacy classes of $G^F$ that 
are contained in the set $C^F$ (see, e.g., \cite[Lemma~2.12]{LiSe}). For 
$a\in A(u_0)$, an element in the corresponding $G^F$-conjugacy class is
given by $u_a=hu_0h^{-1}$ where $h\in G$ is such that $h^{-1}F(h) \in 
C_G(u_0)$ maps to $a$ under the natural homomorphism $C_G(u_0)\rightarrow 
A(u_0)$. (The existence of $h$ is guaranteed by Lang's Theorem; note that
$h$ is not unique but $u_a=hu_0h^{-1}$ is well-defined up to $G^F$-conjugacy.) 
\end{rem}

Let $E\in \Irr(W)^\gamma$ and $\iota_G(E)=(C,\cE)\in \cN_G$. As in 
\S \ref{sub25}, we have $F(C)=C$ and $F^*\cE\cong \cE$. Furthermore, there 
is a certain isomorphism $\psi\colon F^* \cE\stackrel{\sim} {\rightarrow}
\cE$ which induces a map of finite order $\psi_g \colon \cE_g\rightarrow 
\cE_g$ for each $g\in C^F$. Now let us fix an element $u_0\in C^F$ as in 
($\clubsuit$), such that $F$ acts trivially on $A(u_0)$. 

\begin{lem}[Cf.\ Lusztig \protect{\cite[19.7]{Ldisc4}}] \label{lu11} In 
the above setting, let $u_0\in C^F$ be such that ($\clubsuit$) holds. There 
is a natural $A(u_0)$-module structure on the stalk $\cE_{u_0}$. We have 
$\cE_{u_0} \in \Irr(A(u_0))$ and the map $\psi_{u_0}\colon \cE_{u_0} 
\rightarrow \cE_{u_0}$ is given by scalar mulplication with a sign 
$\delta_E=\pm 1$. Furthermore, $Y_E(u_a)= \delta_E\operatorname{Tr}(a,
\cE_{u_0})$ for all $a\in A(u_0)$. 
\end{lem}

Note that $\mbox{Tr}(a,\cE_{u_0})$ is just an entry in the ordinary
character table of $A(u_0)$. In particular, if $a=1$, then $u_1$ is 
$G^F$-conjugate to $u_0$ and so $Y_E(u_0)=\delta_E\dim \cE_{u_0}$. 

\begin{proof} The $A(u_0)$-module structure on $\cE_{u_0}$ is explained
in the proof of \cite[19.7]{Ldisc4}; this also shows that $\psi_{u_0}
\circ a=a\circ \psi_{u_0}\colon \cE_{u_0}\rightarrow \cE_{u_0}$ 
for all $a\in A(u_0)$. (Recall from ($\clubsuit$) that $F$ acts trivially 
on $A(u_0)$.) Since $\cE_{u_0} \in \Irr(A(u_0))$, we conclude that 
$\psi_{u_0}$ acts as a scalar times the identity on $\cE_{u_0}$; let us 
denote this scalar by $\delta_E\in \overline{\Q}_\ell$. Then, for any 
$g\in C^F$, we have $Y_E(g)=\delta_E \mbox{Tr}(a,\cE_{u_0})$ where 
$a\in A(u_0)$ is such that $g$ is $G^F$-conjugate to $u_a$; see 
\cite[19.7]{Ldisc4}. Since $\psi_{u_0}\colon \cE_{u_0} \rightarrow 
\cE_{u_0}$ is a map of finite order, the scalar $\delta_E$ is a root of 
unity. Since the values of the almost characters $R_{E}$ are 
in $\Q$ (see \S \ref{sub21}), and since $R_{E}(u_1)=
q^{d_E}Y_E(u_1)=q^{d_E}\delta_E\mbox{Tr}(1,\cE_{u_0})$ (see \S \ref{sub25}), 
we conclude that we also have $\delta_E\in \Q$. Hence, we finally see 
that $\delta_E=\pm 1$.
\end{proof}

\textit{Thus, the problem of computing the Green functions $Q_w$ 
is further reduced to the determination of the signs $\delta_E=
\pm 1$ for $E \in \Irr(W)^\gamma$} (cf.\ Shoji \cite[1.3, p.~161]{S6a}). 

\begin{exmp} \label{beysp} Let $G$ be of type $E_8$ and $p>5$. Let $C$ be 
the unipotent class denoted by $D_8(a_3)$ in Mizuno \cite{Miz2}, or by 
$E_8(b_6)$ in Carter \cite[p.~407]{C2}. We have $\dim C=28$ and, 
up to conjugation within $G^F$, there is a unique $u_0\in C^F$ such that
$|C_G(u_0)^F|=6q^{28}$; see \cite[p.~455]{Miz2}. We have $A(u_0)\cong \fS_3$ 
and $F$ acts trivially on $A(u_0)$. Let $E\in \Irr(W)$ be the irreducible 
representation denoted by $\phi_{840,13}$ in \cite[\S 13.2]{C2}, or by $840_z$
in \cite[4.13.1]{L1}. Then $\iota_G(E)=(C,\cE)$ where the irreducible 
$A(u_0)$-module $\cE_{u_0}$ corresponds to the sign representation of 
$\fS_3$; see \cite[p.~432]{C2}. It is shown by Beynon--Spaltensetin
\cite[\S 3, Case~5]{BeSp} that 
\[ \delta_E=\left\{\begin{array}{rl} 1 & \quad \mbox{if $q\equiv 1 \bmod 
3$},\\ -1 & \quad \mbox{if $q\equiv -1 \bmod 3$}.\end{array}\right.\]
In particular, there do exist cases in which $\delta_E=-1$.
\end{exmp}

Returning to the general setting, the following corollary interprets the signs
$\delta_E=\pm 1$ somewhat more directly in terms of the character sheaves
$A_E$ in \S \ref{sub23}.

\begin{cor} \label{lu11a} Let $E\in\Irr(W)^\gamma$ and $\iota_G(E)=(C,\cE)
\in\cN_G$. Let $u_0\in C^F$ be as in ($\clubsuit$). Then the isomorphism 
$\phi_{A_E}\colon F^*A_E\stackrel{\sim}{\rightarrow} A_E$ in 
\S \ref{sub23} induces the scalar multiplication by $\delta_E\,q^{d_E}$ on 
the stalk $\cH_{u_0}^i(A_E)$, where $i=-\dim C-\dim T_0$. 
\end{cor}

\begin{proof} This is just a reformulation of Lemma~\ref{lu11}, noting
that $\cH^i(A_E)|_C\cong \cE$ and $\cH^j(A_E)|_C=0$ for $j\neq i$; see
\S \ref{sub24}.
\end{proof}

\begin{rem} \label{replace} Let $d\geq 1$ be the order of the automorphism 
$\gamma\colon W\rightarrow W$; we set 
\[ \cM:=\{r\in \Z_{\geq 1}\mid r \equiv 1 \bmod d\}.\]
Let $E\in \Irr(W)^\gamma$ and $\iota_G(E)=(C,\cE)$ where $C$ is
$F$-stable and $F^*\cE\cong \cE$. Let $u_0\in C^F$ be such that $F$ acts 
trivially on $A(u_0)$; by Lemma~\ref{lu11}, we have a corresponding sign 
$\delta_E=\pm 1$ such that
\[ R_{E}(u_0)=q^{d_E}\delta_E \dim \cE_{u_0} \qquad\mbox{(see
\S \ref{sub25})}.\]
Now let $r\in \cM$ and replace $F$ by $F^r$. Thus, since $r\equiv 1 \bmod d$,
the automorphism of $W$ induced by $F^r$ is again given by~$\gamma$. 
Hence, we can use the chosen map $\sigma_E \colon E\rightarrow E$ (see 
\S \ref{sub21}) to define a corresponding almost character of $G^{F^r}$,
which we denote by $R_{E}^{(r)}$. Finally, we still have $u_0\in C^{F^r}$ 
and $F^r$ acts trivially on $A(u_0)$. Hence, we also have a corresponding 
sign $\delta_E^{(r)}=\pm 1$ as in Lemma~\ref{lu11}, such that
\[ R_{E}^{(r)}(u_0)=q^{d_Er}\delta_E^{(r)} \dim \cE_{u_0}.\]
The following result relates $\delta_E$ and $\delta_E^{(r)}$.
\end{rem}

\begin{thm} \label{thmsplit} With the above notation, we have
$\delta_E^{(r)}=(\delta_E)^r$ for all $r\in \cM$.
\end{thm}

\begin{proof} We use the interpretation of $\delta_E$ in 
Corollary~\ref{lu11a}. Starting with the isomorphism $\phi:=\phi_{A_E} 
\colon F^*A_E \stackrel{\sim}{\rightarrow} A_E$ in 
\S \ref{sub23}, we obtain natural isomorphisms
\[ (F^*)^{j-1}(\phi)\colon (F^*)^jA_E \stackrel{\sim}{\rightarrow} 
(F^*)^{j-1}A_E \qquad \mbox{for $1\leq j \leq r$}.\]
These give rise to an isomorphism 
\[\tilde{\phi}^{(r)}:=\phi\circ F^*(\phi) \circ \ldots \circ 
(F^*)^{r-1}(\phi)\colon (F^*)^rA_E\stackrel{\sim}{\rightarrow} A_E.\] 
We also have a canonical isomorphism $(F^*)^rA_E\cong (F^r)^* A_E$
which, finally, induces an isomorphism
\[\phi^{(r)} \colon (F^r)^*A_E\stackrel{\sim}{\rightarrow} A_E\qquad
\mbox{(see \cite[1.1]{S2})}.\]
We denote the corresponding characteristic function of $A_E$ by 
$\chi^{(r)} \colon G^{F^r}\rightarrow \overline{\Q}_\ell$. Thus, we have 
\[ \chi^{(r)}(g)=\sum_i (-1)^i\mbox{Tr}\bigl(\phi_{i,g}^{(r)},\cH_g^i
(A_E)\bigr) \qquad\mbox{for $g\in G^{F^r}$}.\]
Now assume that $g$ is an element in $G^{F}$, and not just in $G^{F^r}$. 
Then we have 
\[ \phi^{(r)}_{i,g}=(\phi_{i,g})^r \colon \cH_g^i(A_E)
\rightarrow \cH_g^i(A_E)\qquad \mbox{for all $i$};\] 
see \cite[1.1]{S2}. If we take $g=u_0$ (the chosen element in $C^{F}
\subseteq C^{F^r}$) and let $i=-\dim C-\dim T_0$, then $\phi_{i,u_0}\colon 
\cH_{u_0}^i(A_E) \rightarrow \cH_{u_0}^i(A_E)$ is given by scalar 
multiplication with $\delta_E\,q^{d_E}$; see Corollary~\ref{lu11a}. So we 
conclude that 
\[ \phi^{(r)}_{i,u_0}=\mbox{ scalar multiplication by $(\delta_E\, q^{d_E})^r
=(\delta_E)^rq^{d_Er}$ on $\cH_{u_0}^i(A_E)$}.\]
Hence, again in view of Corollary~\ref{lu11a}, it remains to show that
the isomorphism $\phi^{(r)}\colon (F^r)^*A_E\stackrel{\sim}{\rightarrow} 
A_E$ constructed above is the particular isomorphism singled out 
in the discussion in \S \ref{sub23}. But this has already been checked
essentially by Shoji in \cite[2.18.1]{S2}. For this purpose, we have
to consider once more the decomposition 
\[ K\cong \bigoplus_{E \in \Irr(W)} V_E \otimes A_E; \qquad
\mbox{(see \S \ref{sub23})}.\]
As above, starting with the isomorphism $\varphi \colon 
F^*K\stackrel{\sim}{\rightarrow} K$, we obtain a natural isomorphism 
$\varphi^{(r)} \colon (F^r)^*K \stackrel{\sim}{\rightarrow} K$. Then we 
have a unique linear map $\psi_E^{(r)} \colon V_E\rightarrow V_E$ such 
that $\phi^{(r)}\otimes \psi_E^{(r)}$ corresponds to $\varphi^{(r)}$ under
the above direct sum decomposition.  By an argument completely analogous
to that in \cite[3.7]{pbad}, we see that $\psi_E^{(r)}=(\psi_E)^r$,
where $\psi_E \colon V_E\rightarrow V_E$ is determined by $\varphi$
and $\phi_{A_E}$ as in \S \ref{sub23}. Since $r \equiv 1\bmod d$ and $\psi_E$
has order dividing~$d$, we conclude that $\psi_E^r=\psi_E$. Thus, 
$\phi^{(r)}$ also satisfies the requirements in \S \ref{sub23}.
\end{proof}

If $\tilde{\gamma}=\mbox{id}_G$ and $F=F_p^m$, then Theorem~\ref{thmsplit} 
shows that the determination of $\delta_E$ is reduced to the case
where $m=1$. This is exploited in \cite{small} to compute the values of 
Green functions for groups of exceptional type in small characteristic. 

\begin{rem} \label{remtwist} Assume that $F=\tilde{\gamma}\circ F_0=
F_0\circ \tilde{\gamma}$ where $F_0:=F_p^m$ as above and 
$\tilde{\gamma} \colon G \rightarrow G$ is non-trivial. Let 
$E\in \Irr(W)^\gamma$ and $\iota_G(E)=(C,\cE)\in \cN_G$ where $C$ is 
$F$-stable and $F^*\cE\cong \cE$. Assume that there exists an element 
$u_0\in C$ such that
\begin{equation*}
 F_p(u_0)=u_0=\tilde{\gamma}(u_0) \quad \mbox{and} \quad \mbox{$F,F_0$ act
trivially on $A(u_0)$}.\tag{$\clubsuit^\prime$}
\end{equation*}
Then, by Lemma~\ref{lu11}, we have signs $\delta_E=\pm 1$ (with respect to 
$F$) and $\delta_E^\circ=\pm 1$ (with respect to $F_0=F_p^m$). We claim that
\[ \delta_E\delta_E^\circ=\pm 1 \mbox{ does not depend on $m$}.\]
(This remark is used, for example, in the determination of Green functions 
for groups of type ${^2\!E}_6$ in \cite[\S 7]{small}.) A proof is as 
follows. As noted in \cite[2.17]{S2}, we have $\tilde{\gamma}^* A_E\cong 
A_E$. Let $\nu_{A_E}\colon \tilde{\gamma}^* A_E \stackrel{\sim}{\rightarrow} 
A_E$ be an isomorphism. Then $\nu_{A_E}$ induces linear maps $\nu_{A_E,i,u}
\colon \cH_{u_0}^i(A_E)\rightarrow \cH_{u_0}^i(A_E)$ for all~$i$. Let 
$\phi_{A_E}^\circ \colon F_0^*A_E \stackrel{\sim}{\rightarrow} A_E$ be an 
isomorphism as in \S \ref{sub23}. Since $F_0^*(\tilde{\gamma}^*A) \cong 
(F_0\circ \tilde{\gamma})^*A=F^*A$, we obtain 
\[ \phi_{A_E}:=\phi_{A_E}^\circ \circ F_0^*(\nu_{A_E})\colon F^*A_E 
\stackrel{\sim}{\rightarrow} A_E\]
where $F_0^*(\nu_{A_E})\colon F_0^*(\tilde{\gamma}^*A_E) 
\stackrel{\sim}{\rightarrow} F_0^*A_E$ is induced by $\tilde{\gamma}^*
A_E\cong A_E$. Replacing $\nu_{A_E}$ by a scalar multiple if necessary,
we can assume that the above isomorphism satisfies the requirements in
\S \ref{sub23} (see also \cite[2.1, 2.17]{S2}). Now consider stalks at 
$u_0\in C$. Since $F_p(u_0)=u_0=\tilde{\gamma}(u_0)$, the above map
$\nu_{A_E,i,u_0}$ agrees with the map induced by $(F_p^m)^*(\nu_{A_E})$
on $\cH_{u_0}^i(A_E)$. Hence, we obtain that
\[ \phi_{A_E,i,u_0}=\phi_{A_E,i,u_0}^\circ \circ \nu_{A_E,i,u_0} \colon 
\cH_{u_0}^i(A_E)\rightarrow \cH_{u_0}^i(A_E).\]
Now let $i=-\dim C-\dim T_0$. By Corollary~\ref{lu11a},
$\phi_{A_E,i,u_0}$ is given by scalar multiplication with
$\delta_E\,q^{d_E}$; similarly, $\phi_{A_E,i,u_0}^\circ$ is given by scalar 
multiplication with $\delta_E^\circ\,q^{d_E}$. Hence, $\nu_{A_E,i,u_0}$ must 
also be given by scalar multiplication with a sign 
$\delta_E^{\tilde{\gamma}}=\pm 1$, such that $\delta_E=
\delta_E^{\tilde{\gamma}}\, \delta_E^\circ$. Note that 
$\delta_E^{\tilde{\gamma}}$ only depends on $\tilde{\gamma}$.
\end{rem}

\section{Determining the coefficients $p_{E',E}$} \label{sec3}

We keep the notation from the previous section. Taking into account
the information  in \S \ref{sub26} and the orthogonality relations
for Green functions, Lusztig \cite[\S 24.4]{L2e} has described a purely 
combinatorial algorithm for determining the coefficients $p_{E',E}$, which
modifies and simplifies an earlier algorithm of Shoji \cite[\S 5]{S1}.

We say that $w,w'\in W$ are $\gamma$-conjugate, and write $w\sim_\gamma w'$, 
if there exists some $x\in W$ such that $w'=x^{-1}w\gamma(x)$. This defines 
an equivalence relation on $W$; the equivalence classes are called
$\gamma$-conjugacy classes. For $w\in W$, we set 
\[ C_W^\gamma(w):=\{x\in W\mid  w=x^{-1}w\gamma(x)\}.\]
Then the size of the $\gamma$-conjugacy class of $w$ is given by
the index of $C_W^\gamma(w)$ in $W$. By \cite[Prop.~3.3.6]{C2} and 
\cite[Prop.~7.6.2]{C2}, the Green functions of $G^F$ satisfy the following 
orthogonality relations, for any $w,w'\in W$:
\[ \sum_{g \in G_{\text{uni}}^F} Q_w(g)Q_{w'}(g)=\left\{
\begin{array}{cl} [G^F:T_w^F] |C_W^\gamma(w)| & \quad\mbox{if 
$w \sim_\gamma w'$},\\ 0 & \quad \mbox{otherwise}.\end{array}\right.\]
We define the matrix $\tilde{\Omega}=(\tilde{\omega}_{E',E})_{E',E 
\in \Irr(W)^\gamma}$ where 
\[\tilde{\omega}_{E',E}:=\frac{1}{|W|} \sum_{w\in W} [G^F:T_w^F]\, 
\mbox{Tr}(\sigma_{E'}{\circ} w, E') \mbox{Tr}(\sigma_E{\circ} w,E)\in\Q;\]
here, $T_w\subseteq G$ denotes an $F$-stable maximal torus obtained from 
$T_0$ by twisting with~$w$ and the maps $\sigma_E\colon E\rightarrow E$, 
$\sigma_{E'}\colon E'\rightarrow E'$ are as in \S \ref{sub21}. As in 
\cite[3.19]{cbms}, we have the formula
\[ \sum_{E\in \Irr(W)^\gamma} \mbox{Tr}(\sigma_E{\circ} w)
\mbox{Tr}(\sigma_E{\circ} w',E)=\left\{\begin{array}{cl} |C_W^\gamma(w)| & 
\quad \mbox{if $w\sim_\gamma w'$},\\ 0 & \quad \mbox{otherwise}.
\end{array}\right.\]
Note also that, for $E\in \Irr(W)^\gamma$, the function $w\mapsto 
\mbox{Tr}(\sigma_E{\circ} w,E)$ is constant on $\gamma$-conjugacy classes 
(by the definition of $\sigma_E$). 
Combining this with the formulae in \S \ref{sub21}, the above 
orthogonality relations can be restated as follows:
\[ \sum_{g\in G_{\text{uni}}^F} R_{E'}(g) R_E(g)=\tilde{\omega}_{E',E}
\qquad \mbox{for all $E',E\in\Irr(W)^\gamma$}. \]

\begin{lem} \label{lem31} We have $\det(\tilde{\Omega})=\prod_w [G^F:T_w^F]$
where $w$ runs over a set of representatives of the $\gamma$-conjugacy
classes of~$W$.
\end{lem}

\begin{proof} Let $w_1,\ldots,w_n\in W$ be representatives of 
the $\gamma$-conjugacy classes of $W$. It is known that 
then $|\Irr(W)^\gamma|=n$; see, e.g., \cite[Lemma~7.3]{GKP}. Let
us write $\Irr(W)^\gamma=\{E_1,\ldots,E_n\}$. We define a matrix
\[ X=(x_{ii'})_{1\leq i,i'\leq n} \qquad\mbox{where}
\qquad x_{ii'}:=\mbox{Tr}(\sigma_{E_{i'}} {\circ} w_i,E_{i'}).\]
Then, by \S \ref{sub21}, we have $Q_{w_i}(g)=\sum_{1\leq i'\leq n} 
x_{ii'} R_{E_{i'}}(g)$ for all $g\in G_{\text{uni}}^F$ and so
\[ \sum_{g \in G_{\text{uni}}^F} Q_{w_i}(g)Q_{w_j}(g)=
\sum_{1\leq i',j'\leq n} x_{ii'}x_{jj'} \tilde{\omega}_{E_{i'},E_{j'}}=
\bigl( X\cdot \tilde{\Omega}\cdot X^{\text{tr}}\bigr)_{ij}.\]
On the other hand, the left hand side equals $[G^F:T_{w_i}^F]
|C_W^\gamma(w_i)|$ if $i=j$, and $0$ otherwise. Hence, we find that
\[ \det(X\cdot X^{\text{tr}})\det(\tilde{\Omega})=\det(X\cdot 
\tilde{\Omega}\cdot X^{\text{tr}})=\prod_{1\leq i\leq n} [G^F:T_{w_i}^F]
|C_W^\gamma(w_i)|.\]
It remains to show that $\det(X\cdot X^{\text{tr}})=\prod_{1\leq i\leq n} 
|C_W^\gamma(w_i)|$. But this immediately follows from the identity
\[ (X\cdot X^{\text{tr}})_{ij}=\sum_{1\leq i'\leq n} 
\mbox{Tr}(\sigma_{E_{i'}}{\circ} w_i,E_{i'}) \mbox{Tr}(\sigma_{E_{i'}}
{\circ} w_j,E_{i'})=\left\{\begin{array}{cl} |C_W^\gamma(w_i)| & 
\mbox{if $i=j$},\\ 0 & \mbox{if $i\neq j$},\end{array}\right.\]
which holds by the above-mentioned formula from \cite[3.19]{cbms}.
\end{proof}

Following \cite[24.3.4]{L2e}, we define three matrices
\[ P=(p_{E',E}),\qquad  \Omega= (\omega_{E',E}), \qquad 
\Lambda=(\lambda_{E',E}), \] 
where, in each case, the indices run over all $E',E\in\Irr(W)^\gamma$. Here, 
$p_{E',E}$ are the coefficients in \S \ref{sub26}; furthermore, 
\[\omega_{E',E}:=q^{-d_E-d_{E'}} \tilde{\omega}_{E',E}\qquad\mbox{and}
\qquad \lambda_{E',E}:=\sum_{g\in G_{\text{uni}}^F} Y_{E'}(g)Y_E(g).\]
(The integers $d_E$ are defined in \S \ref{sub24}.) 

\begin{prop}[Lusztig \protect{\cite[24.4]{L2e}}] \label{lem32}
We have $P^{\operatorname{tr}}\cdot \Lambda\cdot P=\Omega$. Furthermore, 
for all $E',E\in\Irr(W)^\gamma$, we have $p_{E',E}\in \Z$, 
$\lambda_{E',E}\in \Z$, and $\omega_{E',E}\in \Z$.
\end{prop}

\begin{proof} Recall from \S \ref{sub26} the following relations:
\[ R_{E}|_{G_{\text{uni}}^F}=\sum_{E'\in \Irr(W)^\gamma}  q^{d_E}\,
p_{E',E} Y_{E'}\qquad\mbox{for all $E\in \Irr(W)^\gamma$}. \]
This immediately implies the above matrix identity. The fact that 
$p_{E',E}\in \Z$ was already mentioned in \S \ref{sub26}. The fact
that $\lambda_{E',E}\in \Z$ follows from the fact that the $Y$-functions 
are integer-valued; see \S \ref{sub25}. Finally, the above matrix
identity implies that we also have $\omega_{E',E}\in \Z$.
\end{proof}

We obtain further information about the matrices $P$ and $\Lambda$ by 
taking into account the additional information on $p_{E',E}$ in 
\S \ref{sub26}. Let $E,E'\in\Irr(W)^\gamma$ and $\iota_G(E)=(C,\cE)\in 
\cN_G$, $\iota_G(E')=(C',\cE')\in\cN_G$ where $C,C'$ are $F$-stable and 
$F^*\cE\cong \cE$, $F^*\cE'\cong \cE'$. As in \cite[24.1]{L2e}, we write 
$E\sim E'$ if $C=C'$. This gives rise to a partition
\[ \Irr(W)^\sigma=\cI_1\sqcup \ldots \sqcup \cI_h\]
where $\cI_1,\ldots,\cI_h$ are the equivalence classes for the 
relation~$\sim$. Note that $d_E=d_{E'}$ if $E\sim E'$. Thus, we can
define $d_i:=d_E$ where $E\in \cI_i$. 

\begin{rem} \label{rem33} We fix a labelling of the equivalence
classes $\cI_1,\ldots,\cI_h$ such that 
\[ d_1\geq d_2\geq \ldots \geq d_h,\]
and enumerate $\Irr(W)^\sigma$ in a way which is compatible with the 
above partition of $\Irr(W)^\sigma$. Then it is clear that $\Lambda$ has 
a block diagonal shape, where the blocks correspond to the sets
$\cI_1,\ldots,\cI_h$ (see \cite[24.3.2]{L2e}). Furthermore, $P$ has 
an upper block triangular shape with identity matrices on the diagonal (see
\cite[24.2.10, 24.2.11]{L2e}). More precisely, we can write:
\[ P=\left[\begin{array}{c@{\hspace{1mm}}c@{\hspace{3mm}}c@{\hspace{3mm}}
c@{\hspace{3mm}}} I_{e_1} & P_{1,2} & \cdots & P_{1,h} \\ 0 & I_{e_2} & 
& \vdots \\ \vdots & & \ddots & P_{h{-}1,h} \\ 0  & \cdots & 0 & I_{e_h}
\end{array}\right] \qquad \mbox{and}
\qquad \Lambda=\left[\begin{array}{c@{\hspace{3mm}}c@{\hspace{3mm}}
c@{\hspace{3mm}}c@{\hspace{3mm}}} \Lambda_{1} & 0 & \cdots & 0 \\
0 & \Lambda_{2} & & \vdots \\ \vdots & & \ddots & 0 \\ 0 & \cdots
& 0 & \Lambda_{h} \end{array}\right]; \]
here, $e_j=|\cI_j|$ and $I_{e_j}$ denotes the identity matrix of
size $e_j$. For $1\leq i<j\leq h$, the block $P_{i,j}$ has size $e_i
\times e_j$ and entries $p_{E',E}$ for $E'\in \cI_i$ and $E\in \cI_j$; 
similarly, the block $\Lambda_j$ has size $e_j\times e_j$ and entries 
$\lambda_{E',E}$ for $E',E \in \cI_j$. 

With the additional requirement that $P$ and~$\Lambda$ have block shapes 
as above, it easily follows that $P,\Lambda$ are uniquely determined by 
$\Omega$ and the equation $P^{\text{tr}}\cdot \Lambda \cdot P=\Omega$; 
see \cite[24.4]{L2e}, \cite[\S 5]{S1}  (or the proof of Lemma~\ref{redp}
below). 
\end{rem}

With these preparations, we can now establish a first step towards the 
proof of the Congruence Condition in Section~1. Let $d\geq 1$ be the 
order of $\gamma\colon W\rightarrow W$ and
\[ \cM:=\{r\in \Z_{\geq 1}\mid r \equiv 1 \bmod d\}\]
as in Remark~\ref{replace}. Let $r\in \cM$ and replace $F$ by $F^r$.
We obtain analogous matrices as above, which we now denote by
\[ P^{(r)}=\bigl(p_{E',E}^{(r)}\bigr),\qquad  \Omega^{(r)}= 
\bigl(\omega_{E',E}^{(r)}\bigr), \qquad \Lambda^{(r)}=
\bigl(\lambda_{E',E}^{(r)}\bigr) \] 
and where the indices run again over all $E',E\in\Irr(W)^\gamma$. (Note
that $\gamma$ is also the automorphism of $W$ induced by~$F^r$, since
$r\in\cM$.) 

\begin{lem} \label{redp} Let $r\in \cM$ be a prime such that 
\[r \nmid \det (\Omega) \qquad\mbox{and}\qquad \omega^{(r)}_{E',E}
\equiv \omega_{E',E} \bmod r\quad\mbox{for all $E',E\in \Irr(W)^\gamma$}.\]
Then $p^{(r)}_{E',E}\equiv p_{E',E} \bmod r$ and $\lambda^{(r)}_{E',E}
\equiv \lambda_{E',E} \bmod r$ for all $E',E\in\Irr(W)^\gamma$.
\end{lem}

\begin{proof} For $a\in \Z$, we denote by $\bar{a} \in
\F_r$ the reduction modulo~$r$. If $A=(a_{ij})$ is a matrix with entries 
in $\Z$, then we denote $\bar{A}=(\bar{a}_{ij})$. With this notation, we 
must show that $\bar{P}=\bar{P}^{(r)}$ and $\bar{\Lambda}=
\bar{\Lambda}^{(r)}$ under the given assumptions, which mean that 
\[  \bar{\Omega}^{(r)}=\bar{\Omega}\qquad \mbox{and} \qquad  
\det(\bar{\Omega}^{(r)})=\det(\bar{\Omega})\neq 0.\]
Now note that the Springer correspondence $\iota_G\colon \Irr(W)\rightarrow 
\cN_G$ does not depend on any Frobenius map. Hence, the partition 
$\Irr(W)^\gamma=\cI_1\sqcup \ldots \sqcup \cI_h$ and the block structure 
of $P$, $\Lambda$ in Remark~\ref{rem33} remain the same for $P^{(r)}$, 
$\Lambda^{(r)}$. Thus, $\Lambda$ and $\Lambda^{(r)}$ are block diagonal
matrices; furthermore, $P$ and $P^{(r)}$ are upper block triangular matrices
with identity blocks along the diagonal and so $\det(\bar{P})=
\det(\bar{P}^{(r)})=1$. Consequently, we have $\det(\bar{\Lambda})=
\det(\bar{\Omega})\neq 0$ and $\det(\bar{\Lambda}^{(r)})=
\det(\bar{\Omega}^{(r)})=\det(\bar{\Omega}) \neq 0$. So all of the above 
matrices are invertible. Hence, from the identities
\[ (\bar{P}^{(r)})^{\text{tr}}\cdot \bar{\Lambda}^{(r)}\cdot
\bar{P}^{(r)}=\bar{\Omega}^{(r)}=\bar{\Omega}=\bar{P}^{\text{tr}}\cdot
\bar{\Lambda} \cdot \bar{P}\]
we can deduce the identiy 
\[  (\bar{P}^{\text{tr}})^{-1}\cdot (\bar{P}^{(r)})^{\text{tr}}=
\bar{\Lambda}\cdot \bar{P}\cdot (\bar{P}^{(r)})^{-1}\cdot 
(\bar{\Lambda}^{(r)})^{-1}.\]
Since the block shape remains the same when passing from $F$ to $F^r$, the
left hand side of the above identity is a block lower triangular matrix 
with identity blocks along the diagonal, while the right hand side is a 
block upper triangular matrix. Hence, we conclude that $\bar{P}=
\bar{P}^{(r)}$ and $\bar{\Lambda}=\bar{\Lambda}^{(r)}$, as desired.
\end{proof}

\begin{rem} \label{polp} Arguing as in the first part of the
proof of \cite[Theorem~24.8]{L2e}, one sees that there are well-defined
polynomials $\pi_{E',E}\in\Q[\bq]$ (where $\bq$ is an indeterminate) such 
that $p_{E',E}^{(r)}=\pi_{E',E}(q^r)$ for all $r\in\cM$. With a little
extra work, one can show that $\pi_{E',E}\in\Z[\bq]$. This would, of 
course, also imply the conclusion of Lemma~\ref{redp}, as far as the 
$p_{E',E}$  and $p_{E',E}^{(r)}$ are concerned. 
\end{rem}

\section{Proof of the Congruence Condition} \label{sec4}

As remarked in \cite[p.~202]{HaTu}, it is sufficient to prove the
Congruence Condition in Section~1 in the case where $G$ is simple of 
adjoint type. We assume for the rest of this section that this is the case. 
Let us begin with an endomorphism $F'\colon G\rightarrow G$ such that some 
power of $F'$ is a Frobenius map. Then $G^{F'}$ is one of the groups
considered by Steinberg \cite[\S 11.6]{St68}.
It has been already shown in \cite[Prop.~7.7]{HaTu} that the Congruence 
Condition holds for $(G,F')$ of type ${^3\!D}_4$, ${^2\!B}_2$, ${^2\!G}_2$, 
${^2\!F}_4$. So it remains to consider the case where $F'=F$ is a Frobenius 
map. Assume now that this is the case. Thus, we have 
\[ F=\tilde{\gamma}\circ F_p^m=F_p^m \circ \tilde{\gamma}\qquad
(m\geq 1)\]
as in Section~\ref{sec2}, where $\tilde{\gamma}\colon G \rightarrow G$ is
a graph automorphism of order $d\in\{1,2,3\}$ leaving $T_0,B_0$ invariant, 
and $F_p\colon G \rightarrow G$ is a Frobenius map corresponding to a split 
$\F_p$-rational structure, such that $F_p(t)=t^p$ for all $t\in T_0$. 
Then the map $\gamma\colon W \rightarrow W$ induced by $\tilde{\gamma}$ 
also has order~$d$. Since groups of type ${^3\!D}_4$ have already been dealt 
with, we may actually assume that $d\in \{1,2\}$. We set 
\[ \cM:=\{r\in \Z_{\geq 1}\mid r \equiv 1 \bmod d\}\qquad\mbox{($d\in 
\{1,2\}$)}\]
as in Remark~\ref{replace}. (Thus, $\cM$ consists of all positive integers 
if $d=1$, and of all odd positive integers if $d=2$.) Let $r\in\cM$. Working 
with both $F$ and $F^r$, we will have to consider two sets of matrices:
\[\renewcommand{\arraystretch}{1.3}\begin{array}{c@{\hspace{5pt}}
c@{\hspace{5pt}}c@{\hspace{5pt}}c@{\hspace{5pt}}c@{\hspace{5pt}}
c@{\hspace{5pt}}c@{\hspace{5pt}}c@{\hspace{5pt}}c}
P&=&\bigl(p_{E',E}\bigr),\qquad  & \Omega&= &\bigl(\omega_{E',E}\bigr), 
\qquad &\Lambda&=&\bigl(\lambda_{E',E}\bigr),\\
P^{(r)}&=&\bigl(p_{E',E}^{(r)}\bigr),\qquad & \Omega^{(r)}&=& 
\bigl(\omega_{E',E}^{(r)}\bigr), \qquad &\Lambda^{(r)}&=&
\bigl(\lambda_{E',E}^{(r)}\bigr),\end{array}\] 
defined as in the previous section. For each $E\in\Irr(W)^\gamma$, we
have an almost character $R_E$ and a $Y$-function $Y_E$ for $G^F$; similarly,
we have an almost character and a $Y$-function for $G^{F^r}$, which we
denote by $R_E^{(r)}$ and $Y_E^{(r)}$, respectively. This notation will be 
used throughout this section.

\begin{lem} \label{ftor} Let $r\in \cM$ be a prime. Then we have $|G^{F^r}| 
\equiv |G^F| \bmod r$ and $|T_w^{F^r}|\equiv |T_w^F|\bmod r$
for any $w\in W$.
\end{lem}

\begin{proof} Let $X^\vee(T_0)=\mbox{Hom}(k^\times,T_0)$ be 
the co-character group of $T_0$ (a free abelian group of rank equal to 
$\dim T_0$). Then $\tilde{\gamma}$ induces an automorphism 
$\tilde{\gamma}^*\in \mbox{Aut}(X^\vee)$ such that $\tilde{\gamma}^*(\nu)
(x)=\tilde{\gamma}(\nu(x))$ for all $x\in k^\times$ and $\nu\in X^\vee$. 
We can also naturally regard $W$ as a subgroup of $\mbox{Aut}(X^\vee)$;
see \cite[\S 1.9]{C2}. Let $\bq$ be an indeterminate over $\Z$ and define
\[ f_w:=\det\bigl(\bq\,\mbox{id}_{X^\vee}-(\tilde{\gamma}^*)^{-1}{\circ} w
\bigr) \in\Z[\bq] \qquad \mbox{for any $w\in W$}.\]
Then we have $|T_w^F|=f_w(q)$ where $q=p^m$; see \cite[Prop.~3.3.5]{C2}. 
Furthermore, we have $|G^F|=f(q)$, where we define
\[ f:=\bq^{|\Phi^+|}f_1\sum_{w\in W,\gamma(w)=w} \bq^{l(w)}\in \Z[\bq];\] 
see \cite[\S 2.9]{C2}. Now let us replace $F$ by $F^r$, where $r\in \cM$.
Then $F^r=\tilde{\gamma} \circ F_p^{rm}=F_p^{rm}\circ \tilde{\gamma}$. 
Consequently, we obtain that $|T_w^{F^r}|=f_w(q^r)$ and $|G^{F^r}|=f(q^r)$. 
If $r\in \cM$ is a prime, then Fermat's Little Theorem yields the desired 
congruences.
\end{proof}

\begin{lem} \label{lem41} Let $r\in \N$ be a prime such that 
$r\nmid |G^{F^r}|$. Then $\omega_{E',E}^{(r)} \equiv \omega_{E',E} 
\bmod r$, $p_{E',E}^{(r)} \equiv p_{E',E} \bmod r$ and 
$\lambda_{E',E}^{(r)} \equiv \lambda_{E',E} \bmod r$ for all 
$E',E \in\Irr(W)^\gamma$.
\end{lem}

\begin{proof} Since $G$ is non-trivial, the order of $|G^{F^r}|$ is even and
so $r>2$. Hence, we automatically have $r\in \cM$. Let $w_1,\ldots,w_n\in W$ 
be a set of representatives of the $\gamma$-conjugacy classes of $W$. First
we consider the matrix $\tilde{\Omega}$. We have 
\begin{align*}
\tilde{\omega}_{E',E}&=\sum_{1\leq i \leq n} |C_W^\gamma(w_i)|^{-1}[G^F:
T_{w_i}^F] \mbox{Tr}(\sigma_{E'}{\circ} w_i, E')\mbox{Tr}(\sigma_E{\circ} 
w_i,E), \\ \tilde{\omega}_{E',E}^{(r)}&=\sum_{1\leq i \leq n} 
|C_W^\gamma(w_i)|^{-1} [G^{F^r}: T_{w_i}^{F^r}]\mbox{Tr}(\sigma_{E'}{\circ}
w_i, E') \mbox{Tr}(\sigma_E {\circ} w_i,E).
\end{align*}
Since $G^F\subseteq G^{F^r}$, we also have $r\nmid |G^F|$. Hence,
using Lemma~\ref{ftor}, we obtain that 
\[ [G^{F^r}:T_w^{F^r}]\equiv [G^F:T_w^F] \bmod r\qquad\mbox{for all 
$w\in W$}.\]
By \cite[Prop.~3.3.6]{C2}, we have $N_G(T_w^F)/T_w^F\cong C_W^\gamma(w)$.
Since $N_G(T_w)^F\subseteq G^F$, we conclude that $r\nmid |C_W^\gamma(w)|$.
Consequently, we have $\tilde{\omega}_{E',E}^{(r)}\equiv 
\tilde{\omega}_{E',E}\bmod r$. Now 
\[\tilde{\omega}_{E',E}=q^{d_E+d_{E'}}{\omega}_{E',E}\quad \mbox{and}
\quad\tilde{\omega}_{E',E}^{(r)}=q^{(d_E+d_{E'})r}\omega_{E',E}^{(r)}.\]
Using $r\nmid q$ and Fermat's Little Theorem, we conclude that we also 
have ${\omega}_{E',E}^{(r)}\equiv {\omega}_{E',E}\bmod r$. 
Finally, by Lemma~\ref{lem31}, we have $r\nmid \det(\tilde{\Omega})$ and, 
hence, also $r\nmid \det(\Omega)$ (again, since $r\nmid q$). So we can 
apply Lemma~\ref{redp}. 
\end{proof}

\begin{lem} \label{lem42} Let $u\in G^F$ be unipotent and $E\in
\Irr(W)^\gamma$. Let $r\in\N$ be a prime such that 
$r\nmid |G^{F^r}|$. Then $Y_E^{(r)}(u)=Y_E(u)$.
\end{lem}

\begin{proof} As in the previous proof, we have $r>2$ and $r\in \cM$. 
Let $\iota_G(E)=(C,\cE)$ where $C$ is $F$-stable and
$F^*\cong \cE$. If $u\not\in C$, then $Y_E(u)=Y_E^{(r)}(u)=0$. So
now let $u\in C$. Since we are assuming that $G$ is simple
of adjoint type, it is known that there exists an element $u_0\in C^F$
such that $(\clubsuit)$ in Remark~\ref{rem20} holds. For $G$ of classical 
type, such a representative $u_0\in C$ is explicitly described by 
Shoji \cite[\S 2]{S6}; see, e.g., \cite[2.10]{S6} for type ${^2\!D}_n$ in
arbitrary characteristic. If $G$ is of exceptional type, the existence of
$u_0$ is guaranteed by \cite[Lemma~20.16]{LiSe}, the proof of which
involves a certain amount of case--by--case considerations. For
example, for type $E_7,E_8$ one can simply look through the list of 
class representatives determined by Mizuno \cite{Miz2}. See also the
discussion in \cite[\S 2]{Tay} (in good characteristic). In any case,
let us now fix an element $u_0\in C^F$ such that $F$ acts trivially on 
$A(u_0)$. Then $u=gu_0g^{-1}$ for some $g\in G$, and we have $x:=
g^{-1}F(g)\in C_G(u_0)$; we denote by~$a$ the image of~$x$ in $A(u)$. 
Consequently, we have 
\[ Y_E(u)=\delta_E\mbox{Tr}(a,\cE_{u_0}), \qquad \mbox{see
Lemma~\ref{lu11}}.\]
Note that, since the values of the $Y$-functions are integers, the same
is true for $\mbox{Tr}(a,\cE_{u_0})$. Let us now replace $F$ by $F^r$. 
We still have $x':=g^{-1}F^r(g)\in C_G(u_0)$; consequently,
if we denote by~$a'$ the image of $x'$ in $A(u_0)$, then 
\[ Y_E^{(r)}(u)=\delta_E^{(r)}\mbox{Tr}(a',\cE_{u_0})\]
(and, again, these values are integers). By Theorem~\ref{thmsplit}, we 
have $\delta_E^{(r)}=(\delta_E)^r$. Since $r$ is odd, we conclude that 
$\delta_E^{(r)}=\delta_E$. Hence, it remains to show that 
\[\mbox{Tr}(a,\cE_{u_0})=\mbox{Tr}(a', \cE_{u_0}).\]
Now, since $F$ acts trivially on $A(u_0)$, we 
have $A(u_0)\cong C_G(u_0)^F/C_G^\circ (u_0)^F$ (see \cite[p.~33]{C2}). 
Hence, since $r\nmid |G^F|$, we also have $r\nmid |A(u_0)|$. But then a 
well-known result in the character theory of finite groups implies that 
$\mbox{Tr}(a,\cE_{u_0})=\mbox{Tr}(a^r,\cE_{u_0})$. So it will be enough to 
show that $a^r=a'$. This is seen as follows. We have $F(g)=gx$ and $F^r(g)=
gx'$, with $x,x'\in C_G(u_0)$, which implies that 
\[x'=xF(x)F(x)^2\cdots F^{r-1}(x).\]
Since $F$ acts trivially on $C_G(u_0)/C_G^\circ(u_0)$, we have 
$x^{-1}F^i(x)\in C_G^\circ(u_0)$ for all $i\geq 1$. This yields $x'=x^rc$ 
for some $c\in C_G^\circ(u_0)$ and, hence, $a'=a^r$, as desired. 
\end{proof}

We can now complete the proof of the Congruence Condition, as follows. 
Let $u\in G^F$ be unipotent and $T\subseteq G$ be an $F$-stable maximal 
torus. Let $w \in W$ be such that $T$ is obtained from $T_0$ by twisting 
with~$w$ (relative to $F$). Thus, we have $Q_{T,F}=Q_w$. Let $r\in \N$ be 
a prime such that $r\nmid |G^{F^r}|$; then $r>2$ and $r\in \cM$ (see the 
above proofs). As pointed out in \cite[p.~204]{HaTu}, we also have that $T$ 
is obtained by twisting with~$w$ relative to $F^r$. Thus, we have 
$Q_{T,F^r}(u)=Q_w^{(r)}(u)$ where, as usual, we indicate by the superscript 
``$(r)$'' that we mean the Green function for $G^{F^r}$. Now, by
Proposition~\ref{sumup}, we have the following formulae.
\begin{align*}
Q_w(u) &=\sum_{E',E\in \Irr(W)^\gamma} \mbox{Tr}(\sigma_E{\circ} w,E)q^{d_E}
p_{E',E} Y_{E'}(u),\\
Q_w^{(r)}(u) &=\sum_{E',E\in \Irr(W)^\gamma} \mbox{Tr}(\sigma_E{\circ} w,E)
q^{d_Er} p_{E',E}^{(r)} Y_{E'}^{(r)}(u).
\end{align*}
By Lemma~\ref{lem41}, we have $p_{E',E}^{(r)} \equiv p_{E',E} \bmod r$; 
furthermore, by Lemma~\ref{lem42}, we have $Y_E^{(r)}(u)=Y_E(u)$. Finally, 
by Fermat's Little Theorem, we have $q^{d_E}\equiv q^{d_Er} \bmod r$.
Hence, we conclude that $Q_w(u)\equiv Q_w^{(r)}(u)\bmod r$. Thus, the 
Congruence Condition in Section~1 is proved. \qed



\begin{thebibliography}{131}

\bibitem{bbd}
A. A. Beilinson, J. Bernstein and P. Deligne, \textit{Faisceaux pervers}.
Ast\'erisque No. 100, Soc. Math. France, 1982.

\bibitem{BeSp}
W. M. Beynon and N. Spaltenstein, \textit{Green functions of finite Chevalley
groups of type $E_n$ ($n=6,7,8$)}. J. Algebra {\bf 88} (1984), 584--614.

\bibitem{C2} 
R. W.~Carter, \textit{Finite groups of Lie type: Conjugacy classes 
and complex characters}. Wiley, New York, 1985; reprinted 1993 as {W}iley
{C}lassics {L}ibrary {E}dition.

\bibitem{DeLu}
P.~Deligne and G.~Lusztig, \textit{Representations of reductive groups over
finite fields}, Annals Math. {\bf 103} (1976), 103--161.

\bibitem{aver}
M. Geck, \textit{On the average values of the irreducible characters of finite
groups of Lie type on geometric unipotent classes}. Doc.\ Math.\ J.\ DMV
{\bf 1} (1996), 293--317 (electronic).

\bibitem{first}
M. Geck, \textit{A first guide to the character theory of finite groups of
Lie type}.  {\it In:} Local representation theory and simple groups
(eds. R. Kessar, G. Malle, D. Testerman), pp.~63--106, EMS Lecture
Notes Series, Eur. Math. Soc., Z\"urich, 2018.

\bibitem{pbad}
M. Geck, \textit{On the values of unipotent characters in bad characteristic}.
Rend. Cont. Sem. Mat. Univ. Padova (2019), online first, 
\url{DOI: 10.4171/RSMUP/14}.

\bibitem{small}
M. Geck, \textit{On the computation of Green functions in small 
characteristics}, in preparation.

\bibitem{GKP}
M.~Geck, S. Kim and G. Pfeiffer, \textit{Minimal length elements in twisted
  conjugacy classes of finite Coxeter groups}. J. Algebra {\bf 229} (2000),
  570--600.

\bibitem{Glau}
G. Glauberman, \textit{Correspondences of characters for relatively coprime
operator groups}. Canad. J. Math. {\bf 20} (1968), 1456--1488.

\bibitem{HaTu}
B. Hartley and A. Turull, \textit{On characters of coprime operator
groups and the Glauberman character correspondence}. J. reine angew.
Math. {\bf 451} (1994), 175--219.

\bibitem{LiSe}
M. W. Liebeck and G. M. Seitz, \textit{Unipotent and nilpotent classes in 
simple algebraic groups and Lie algebras}. Math. Surveys and Monographs,
vol. 180, Amer. Math. Soc., Providence, RI, 2012.

\bibitem{cbms}
G. Lusztig, \textit{Representations of finite Chevalley groups}. C.B.M.S.\
  Regional Conference Series in Mathematics, vol.~39, Amer. Math. Soc.,
  Providence, RI, 1977.

\bibitem{L1} 
G.~Lusztig, \textit{Characters of reductive groups over a finite field}.
Ann.\ Math.\ Studies {\bf 107}, Princeton U.\ Press, 1984.

\bibitem{LuIC}
G. Lusztig, \textit{Intersection cohomology complexes on a reductive group}.
  Invent. Math. {\bf 75} (1984), 205--272.

\bibitem{L2a} 
G.~Lusztig, \textit{Character sheaves}. Adv.\ Math. {\bf 56} (1985), 
193--237.

\bibitem{L2b} 
G.~Lusztig, \textit{Character sheaves II}. Adv.\ Math. {\bf 57} (1985), 
226--265.

\bibitem{L2c} 
G.~Lusztig, \textit{Character sheaves III}. Adv.\ Math. {\bf 57} (1985), 
266--315.

\bibitem{L2d} 
G.~Lusztig, \textit{Character sheaves IV}. Adv.\ Math. {\bf 59} (1986), 
1--63.

\bibitem{L2e} 
G.~Lusztig, \textit{Character sheaves V}. Adv.\ Math. {\bf 61} (1986), 
103--155.

\bibitem{L5}
G. Lusztig, \textit{Green functions and character sheaves}. Ann. Math.
  {\bf131} (1990), 355--408.

\bibitem{Ldisc4}
G.~Lusztig, \textit{Character sheaves on disconnected groups, IV}.
Represent. Theory {\bf 8} (2004), 145--178.

\bibitem{L10} 
G.~Lusztig, \textit{On the cleanness of cuspidal character sheaves}. 
Mosc. Math. J. {\bf 12} (2012), 621--631. 

\bibitem{Miz2}
K. Mizuno, \textit{The conjugate classes of unipotent elements of the
Chevalley groups $E_7$ and $E_8$}. Tokyo J. Math. {\bf 3} (1980), 391--461.

\bibitem{Nav1}
G. Navarro, \textit{Some open problems on coprime action and character
correspondences}. Bull. London Math. Soc. {\bf 26} (1994). 513--522.

\bibitem{Nav}
G. Navarro, \textit{Character theory and the McKay conjecture}.
Cambridge Studies in advanced mathematics 175, Cambridge Univ. Press, 2018.

\bibitem{S1}
T. Shoji, \textit{Green functions of reductive groups over a finite field}.
In: \textit{The Arcata Conference on Representations of Finite Groups} 
(Arcata, Calif., 1986),  Proc. Sympos. Pure Math., 47, Part 1, Amer. 
Math. Soc., Providence, RI, 1987, pp.~289--302.

\bibitem{S2} 
T.~Shoji, \textit{Character sheaves and almost characters of reductive 
groups}. Adv.\ Math. {\bf 111} (1995),  244--313.

\bibitem{S3} 
T.~Shoji,  \textit{Character sheaves and almost characters of reductive
groups, II}.  Adv.\ Math. {\bf 111} (1995),  314--354.

\bibitem{S6a} 
T. Shoji, \textit{Generalized Green functions and unipotent classes for 
finite reductive groups, I}. Nagoya Math. J. {\bf 184} (2006), 155--198.

\bibitem{S6} 
T. Shoji, \textit{Generalized Green functions and unipotent classes for 
finite reductive groups, II}. Nagoya Math. J. {\bf 188} (2007), 133--170.

\bibitem{St68}
R.~Steinberg, \textit{Endomorphisms of linear algebraic groups}. Mem.
  Amer. Math. Soc., no. 80. Amer. Math. Soc., Providence, R.I., 1968. 

\bibitem{Tay}
J.~Taylor, \textit{On unipotent supports of reductive groups with a 
disconnected centre}. J. Algebra {\bf 391} (2013), 41--61. 
\end{thebibliography}
\end{document}